\documentclass{amsart}

\usepackage{times,url}
\usepackage{amsmath,amssymb,amsfonts}
\usepackage{graphicx}
\usepackage[utf8]{inputenc}

\theoremstyle{definition}

\newtheorem{theorem}{Theorem}[section]
\newtheorem{lemma}{Lemma}[section]
\newtheorem{corollary}{Corollary}[section]
\newtheorem{definition}{Definition}[section]

\title{A Propositional Linear Time Logic with Time Flow Isomorphic to $\omega^2$}

\author{Bojan Marinković}
\address{Mathematical Institute of the Serbian Academy of Sciences and Arts, Serbia}
\email{\url{bojanm@mi.sanu.ac.rs}}

\author{Zoran Ognjanović}
\address{Mathematical Institute of the Serbian Academy of Sciences and Arts, Serbia}

\author{Dragan Doder}
\address{University of Belgrade, Faculty of Mechanical Engineering, Serbia}

\author{Aleksandar Perović}
\address{University of Belgrade, Faculty of Transportation and Traffic Engineering, Serbia}

\begin{document}
\maketitle

\begin{abstract}

Primarily guided with the idea to express zero-time transitions by means of temporal propositional language,
we have developed a temporal logic where the time flow is isomorphic to ordinal $\omega^2$ (concatenation of $\omega$ copies of $\omega$). If we think of
$\omega^2$ as lexicographically ordered $\omega\times \omega$, then any particular zero-time transition can be represented
by states whose indices are all elements of some $\{n\}\times\omega$. In order to express non-infinitesimal transitions, we have introduced
a new unary temporal operator $[\omega]  $ ($\omega$-jump), whose effect on the time flow is the same as the effect of $\alpha\mapsto \alpha+\omega$ in $\omega^2$.
In terms of lexicographically ordered $\omega\times \omega$, $[\omega]  \phi$ is satisfied in $\langle i,j\rangle$-th time instant iff $\phi$ is
satisfied in $\langle i+1,0\rangle$-th time instant. Moreover, in order to formally capture the natural semantics of the until operator $\mathtt U$, we have introduced
a local variant $\mathtt u$ of the until operator. More precisely, $\phi\,\mathtt u \psi$ is satisfied in $\langle i,j\rangle$-th time instant iff $\psi$
is satisfied in $\langle i,j+k\rangle$-th time instant for some nonnegative integer $k$, and $\phi$ is satisfied in $\langle i,j+l\rangle$-th time instant for all
$0\leqslant l<k$.
As in many of our previous publications, the leitmotif is the usage of infinitary inference rules
in order to achieve the strong completeness.

\textbf{Keywords}: Temporal logic, Zero time transitions, Axiomatization, Strong completeness, Decidability 
\end{abstract}

\section{Introduction}

In \cite{GMM99}, A. Gargantini, D. Mandrioli and A. Morzenti have
proposed a general framework for formal description of systems
with so called zero-time transitions, illustrated through Petri
nets as state machines and TRIO as assertion language. The key
novelty in their approach of modeling zero-time transitions was
introduction of infinitesimals in the time flow. More precisely,
they have adopted and operationalized a natural assumption that
transitions of any particular system from one state to the next
are not instantaneous but infinitesimal with respect to the
execution time of the entire system. From the logician's point of
view, we may see \cite{GMM99} as a skilful application of the
interpretation method, particularly if we analyze proofs of the
correctness (or adequacy) of proposed modeling. In the recent
companion paper \cite{FMMR12}, L. Ferrucci, D. Mandrioli, A.
Morzenti and M. Rossiuses use concepts from non-standard analysis
and provide notions of micro and macro steps in an extension of
the TRIO metric temporal general-purpose specification language.

As a natural consequence of our research background and scientific taste, which is in large part focused on probability logic, temporal logic and tame fragments of $L_{\omega_1}$ and $L_{\omega_1,\omega}$
logics (among those are admissible fragments in Barwise sense, see \cite{B75}), inspired by the work presented in \cite{GMM99} we have decided to develop a discrete linear time temporal propositional logic adequate
for modeling zero-time transitions. Following the concept of non-instantaneous transitions of a system and discrete linear time model, we end up with the time flow isomorphic to concatenation of $\omega$ copies
of $\omega$, i.e. with $\omega^2$ as the model of the time flow.

Arguably, the most intuitive representation of the ordinal $\omega^2$ is the lexicographically ordered $\omega\times\omega$. For our purpose, changes of the first coordinate represent different states of a system,
while the changes of the second coordinate represent transitions from one state to the next. Hence, it was natural to introduce the following temporal operators:

\begin{itemize}

\item $[\omega]  $. It represents next state of a system. In the terms of a time flow, it corresponds to the operation $\alpha\mapsto \alpha+\omega$ on $\omega^2$. Semantically,
$[\omega]   \phi$ is satisfied in the $\langle i,j\rangle$-th time instant iff $\phi$ is satisfied in the $\langle i+1,0\rangle$-th time instant;

\item $[1]$. It represent the infinitesimal change of a system within some state. In terms of a time flow, it behaves like the usual next operator: $[1] \phi$ is satisfied in the
$\langle i,j\rangle$-th moment iff $\phi$ is satisfied in the $\langle i,j+1\rangle$-th moment;

\item $\mathtt U$. It represent the adequate generalization of the
until operator from $\omega$ to $\omega^2$. Semantically,
$\phi\,\mathtt U\,\psi$ is satisfied in the $\langle
i,j\rangle$-th moment iff there is $\langle
k,l\rangle\geqslant_{\rm lex}\langle i,j\rangle$ so that $\psi$ is
satisfied in the $\langle k,l\rangle$-th moment, and for all
$\langle i,j\rangle \leqslant_{\rm lex} \langle r,s\rangle<_{\rm
lex}\langle k,l\rangle$, $\phi$ is satisfied in $\langle
r,s\rangle$-th moment. Here $\leqslant_{\rm lex}$ denotes
lexicographical ordering;

\item $\mathtt u$. It is a local version of the until operator. Semantically, $\phi\,\mathtt u\,\psi$ is satisfied in the
$\langle i,j\rangle$-th moment iff there is a nonnegative integer $k$ such that $\psi$ is satisfied in the $\langle i,j+k\rangle$-th moment, and for all $l<k$, $\phi$
is satisfied in $\langle i,j+l\rangle$-th moment.

\end{itemize}
The main technical results are the proofs of the completeness theorem and determination of the complexity for the satisfiability procedure (PSPACE).

\subsection{Related work}

The present paper can be classified as a research related to discrete linear time temporal logics, with particular application on system descriptions and handling zero-time transitions in petri nets.
For modal and temporal part,
we refer the reader to \cite{BAPM83,Bur78,Bur82,Bur79,Bur84,EC82,EH82,Eme90,Fel84,gabbay2,HJ94,LS82,Pnu77,Pri57,Rey01}.
The infinitary techniques presented here (application of infinitary inference rules in order to overcome inherited noncompactness) are connected with our previous research, see \cite{DOM10,DMOPR10,Ogn06}. Decidability argumentation presented here is the modification of the work of A. Sistla and E. Clarke presented in \cite{SC85}.
The motivation for this particular modification of discrete linear time temporal logic has come from the research of A. Gargantini, D. Mandrioli and A. Morzenti that was presented in \cite{GMM99}.

\section{Syntax and semantics}

Let $Var=\{p_n\ |\ n\in\omega\}$ be the set of propositional variables. The set $For$ of all $L([1], [\omega]  ,\mathtt u,\mathtt U)$-formulas
is the smallest superset of $Var$ that is closed under the following formation rules:

\begin{itemize}

\item $\phi\mapsto *\phi$, where $*\in\{\lnot, [1], [\omega]  \}$;

\item $\langle \phi,\psi\rangle\mapsto (\phi * \psi)$, where $*\in \{\land, \mathtt u, \mathtt U\}$.

\end{itemize}
As it is usual, in order to simplify notation we will use the standard convention of omission of parentheses and the standard priority of connectives
(all unary connectives have the greater priority than any binary connective; connectives of the same arity have identical priority).
From now on, by a formula we will mean an $L([1], [\omega]  ,\mathtt u,\mathtt U)$-formula. Formulas will be denoted by $\phi$, $\psi$
and $\theta$, indexed if necessary. The remaining logical and temporal connectives $\lor$, $\to$, $\leftrightarrow$, $\mathtt f$, $\mathtt g$, $\mathtt F$ and $\mathtt G$
are defined in the usual way:

\begin{itemize}

\item $\phi\lor\psi =_{\rm def}\lnot (\lnot \phi \land \lnot \psi)$;

\item $\phi \to \psi =_{\rm def} \lnot \phi \lor\psi$;

\item $ \phi \leftrightarrow \psi =_{\rm def} (\phi\to\psi)\land (\psi\to\phi)$;

\item $\mathtt f\phi =_{\rm def} (\phi\to\phi)\mathtt u \phi$;

\item $\mathtt g\phi =_{\rm def} \lnot\mathtt f\lnot\phi$;

\item $\mathtt F\phi =_{\rm def} (\phi\to\phi)\mathtt U \phi$;

\item $\mathtt G\phi =_{\rm def} \lnot\mathtt F\lnot\phi$;

\item $[a]^0\phi =_{\rm def}\phi$; $[a]^{n+1}\phi=_{\rm def}[a][a]^n\phi$, $a\in \{1,\omega\}$.

\end{itemize}
Nonempty sets of formulas will be called theories.

The time flow is isomorphic to the ordinal $\omega\cdot\omega$ (recall that in ordinal arithmetics $\omega\cdot\omega$ is a concatenation of $\omega$ copies of $\omega$).
Instead of $\langle \omega\cdot\omega,\in\rangle$, as a canonical ordering we will use lexicographically ordered $\omega\times \omega$. Recall that
$
\langle i,j\rangle\leqslant_{\rm lex}\langle k,l\rangle
$
iff $i<k$, or $i=k$ and $j\leqslant l$. If the context is clear, we will omit ``lex" from the subscript. From now on, the elements of $\omega\times \omega$
would be referred as time instants, and will be denoted by $\mathbf r$ , $\mathbf s$ and $\mathbf t$, indexed or primed if necessary.
In particular, $\mathbf r$ should be regarded as $\langle r_1,r_2\rangle$, $\mathbf s$ should be regarded as $\langle s_1,s_2\rangle$ and so on.

A model is any function (evaluation) of the form $\xi:\omega\times\omega\times Var \longrightarrow \{0,1\}$. Models will be denoted by
$\xi$, $\eta$ and $\zeta$, indexed if necessary.

\begin{definition}
\label{satisfiability relation}
Let $\xi:\omega\times\omega\times Var \longrightarrow \{0,1\}$. We define the predicate $\xi\models_{\mathbf r}\phi$, which reads ``a model $\xi$ satisfies formula
$\phi$ in the $\mathbf r$-th moment (or in the $\mathbf r$-th time instant)", recursively on the complexity of formulas as follows:

\begin{enumerate}

\item $\xi\models_\mathbf r p_n \Leftrightarrow_{\rm def} \xi(\mathbf r,p_n)=1$;

\item $\xi\models_\mathbf r\lnot\phi\Leftrightarrow_{\rm def}\xi\not\models_\mathbf r\phi$;

\item $\xi\models_\mathbf r\phi\land\psi \Leftrightarrow_{\rm def}$ $\xi\models_\mathbf r\phi$ and
$\xi\models_\mathbf r\psi$;

\item $\xi\models_\mathbf r[1]\phi\Leftrightarrow_{\rm def} \xi\models_{\langle r_1,r_2+1\rangle}\phi$;

\item $\xi\models_\mathbf r [\omega]  \phi \Leftrightarrow_{\rm def} \xi\models_{\langle r_1+1,0 \rangle}\phi$;

\item $\xi\models_\mathbf r\phi\, \mathtt u \psi\Leftrightarrow_{\rm def}$ there exists $k\in\omega$ such that
$\xi\models_{\langle r_1,r_2+k\rangle}\psi$ and $\xi\models_{\langle r_1,r_2+i\rangle}\phi$ for all $0\leqslant i<k$;

\item $\xi\models_\mathbf r\phi\, \mathtt U\, \psi\Leftrightarrow_{\rm def}$ there exists $\mathbf s\geqslant \mathbf r$
such that $\xi\models_\mathbf s \psi$ and $\xi\models_\mathbf t\phi$ for all $\mathbf r\leqslant \mathbf t<\mathbf s$.
\hfill $\square$

\end{enumerate}

\end{definition}
A formula $\phi$ is satisfiable iff there exist a model $\xi$ and a time-instant $\mathbf t$ so that $\xi\models_\mathbf t\phi$.
A formula $\phi$ is valid iff $\xi\models_\mathbf t\phi$ for all $\xi$ and all $\mathbf t$. For instance, an immediate
consequence of (2) and (3) of Definition \ref{satisfiability relation} is validity of any substitutional instance of any classical tautology.
A slightly less trivial example of a valid formula is $[1] {[\omega]  }\phi \leftrightarrow [\omega]  \phi$. Indeed,

\begin{eqnarray*}
\xi\models_\mathbf t [1] [\omega]   \phi &\Leftrightarrow& \xi\models_{\langle t_1,t_2+1\rangle}[\omega]  \phi\\
&\Leftrightarrow& \xi\models_{\langle t_1+1,0\rangle}\phi\\
&\Leftrightarrow& \xi\models_\mathbf t[\omega]  \phi.
\end{eqnarray*}
If $T$ is a theory, then $\xi\models_\mathbf t T$ means that $\xi\models_\mathbf t\phi$ for all $\phi\in T$, while
$T\models\phi$ means that, for all $\xi$ and all $\mathbf t$, $\xi\models_\mathbf t T$ implies
$\xi\models_\mathbf t \phi$. We say that a theory $T$ is satisfiable iff there exist a model $\xi$ and a time-instant $\mathbf t$
so that $\xi\models_\mathbf t T$. A theory $T$ is finitely satisfiable iff all finite subsets of $T$ are satisfiable.

\begin{theorem}
\label{noncompactness}
The compactness theorem fails for $L([1],[\omega]  ,\mathtt u,\mathtt U)$.
\end{theorem}

\begin{proof}
Let
$$
T=\{\mathtt F\lnot p_0\}\cup \{[\omega]  ^m[1]^n p_0\ |\ m,n\in\omega\}.
$$
If $\xi\models_\mathbf t \mathtt F\lnot p_0$, then there exist $m,n\in\omega$ so that $\xi\models_{\langle t_1+m,t_2+n\rangle}\lnot p_0$.
As a consequence, $\xi\not\models_\mathbf t[\omega]  ^m[1]^n p_0$, so $T$ is unsatisfiable. It remains to show that $T$ is finitely satisfiable.

Let $S$ be a nonempty finite subset of $T$. Since $S$ is finite, there exists a positive integer $k$ such that $k>\max(m,n)$ for all formulas
of the form $[\omega]  ^m[1]^n p_0$ that appears in $S$. Let $\xi$ be any model such that $\xi(\mathbf r,p_0)=1$ for all $\mathbf r<\langle k,0\rangle$
and $\xi(k,0,p_0)=0$. Then, $\xi\models_{\langle 0,0\rangle}S,\mathtt F\lnot p_0$, so $S$ is satisfiable.
\end{proof}

\section{Complete axiomatization}

Since $[1]$ and $[\omega]$ are essentially modal operators with some additional properties (self-duality for example), one natural way to approach the construction of
a syntactical counterpart $\vdash$ of satisfiability relation $\models$ is to determine sufficient but reasonable amount of properties that enables the construction of the standard monster model (the set of worlds is the set of all saturated theories). In particular, one must provide that $T\vdash \phi$ implies $[a]T=\{[a]\psi\ |\ \psi\in T\}\vdash [a]\phi$, $a\in\{1,\omega\}$ and validity of deduction theorem ($T\vdash \phi\to \psi$ iff $T,\phi\vdash \psi$).

However, our situation is significantly simpler than the general one. Firstly, all our models have exactly the same frame $\omega^2$. Secondly, we can syntactically define satisfiability of propositional letters in any node by formulas of the form $[\omega]^m[1]^np$. Hence, if $T$ is any complete theory, then it is quite reasonable to expect that the function $\xi_T:\omega\times\omega\times Var\longrightarrow \{0,1\}$ defined by
$$
\xi_T(m,n,p)=1\ {\rm iff}\ T\vdash [\omega]^m[1]^np
$$
is a model of $T$ in a sense that $\xi_T\models_{\langle m,n\rangle}\phi$ iff $T\vdash [\omega]^m[1]^n\phi$ (in particular, $\xi_T\models_{\langle 0,0\rangle}T$).

So, our definition of $\vdash$ will incorporate sufficient amount of semantical properties of $\models$ to ensure formal proof of the fact that $\xi_T$ described above is a model of a complete theory $T$.

\subsection{Axioms and inference rules}

The axioms of $L([1],[\omega],\mathtt u,\mathtt U)$ are all instances of the following seven schemata:

\begin{enumerate}

\item[A1] Tautology instances;

\item[A2] ${[1] {[\omega] \phi}} \leftrightarrow [\omega]   \phi$;

\item[A3] $\lnot [a]\phi \leftrightarrow [a]\lnot\phi$, $a\in\{1,\omega\}$;

\item[A4] $[a](\phi *\psi)\leftrightarrow ([a]\phi * [a]\psi)$, $[a]\in\{1, \omega\}$ and $*\in\{\land, \lor, \to, \leftrightarrow\}$;

\item[A5] $\psi \to (\phi\,\mathtt u\,\psi)$;

\item[A6] $\phi\,\mathtt u\,\psi\to \phi\,\mathtt U\,\psi$;

\item[A7] $\left(\bigwedge_{k=0}^n [1]^k(\phi\land \lnot \psi)\land [1]^{n+1}\psi\right)\to \phi\,\mathtt u\, \psi$;

\item[A8] $\left(\bigwedge_{k=0}^n [\omega]^k \mathtt g(\phi\land \lnot \psi)\land [\omega]^{n+1}\phi\,\mathtt u\,\psi\right)\to \phi\,\mathtt U\, \psi$.

\end{enumerate}

A1 reflects the fact that all tautology instances are valid. A2 captures the interplay between $[1]$ and $[\omega]$, which in ordinal arithmetics can be stated as $1+\omega=\omega$. A3 and A4 reflect self-duality of both $[1]$ and $[\omega]$. A5, A6, A7 and A8 capture the $\Rightarrow$ part in the following characterization of $\mathtt u$ and $\mathtt U$:

\begin{itemize}

\item $\lnot(\phi\,\mathtt u\,\psi) \Leftrightarrow \lnot\psi \land \bigwedge_{n\in\omega}\bigvee_{k=0}^n [1]^k(\lnot \phi \lor \psi)\lor [1]^{n+1}\lnot\psi$;

\item $\lnot(\phi\,\mathtt U\,\psi) \Leftrightarrow \lnot(\phi\,\mathtt u\,\psi) \land \bigwedge_{n\in\omega}\bigvee_{k=0}^n[\omega]^k\lnot\mathtt g(\phi\land \lnot\psi)\lor [\omega]^{n+1}\lnot(\phi\,\mathtt u\,\psi)$.

\end{itemize}

The inference rules of $L([1],[\omega],\mathtt u, \mathtt U)$ are the following four rules:

\begin{itemize}

\item[R1] Modus ponens: from $\phi$ and $\phi\to \psi$ infer $\psi$;

\item[R2] Necessitation: from $\phi$ infer $[a]\phi$, $a\in\{1,\omega\}$;

\item[R3] $\mathtt u$-rule: from the set of premises
$$
\{\theta\to\lnot \psi\}\cup\left\{\theta\to\bigvee_{k=0}^n [1]^k(\lnot \phi \lor \psi)\lor [1]^{n+1}\lnot\psi\ |\ n\in\omega\right\}
$$ infer $\theta \to \lnot(\phi\,\mathtt u\,\psi)$;

\item[R4] $\mathtt U$-rule: from the set of premises
$$
\{\theta\to\lnot(\phi\,\mathtt u\,\psi)\}\cup\left\{\theta\to\bigvee_{k=0}^n [\omega]^k\lnot \mathtt g(\phi \land \lnot\psi)\lor [\omega]^{n+1}\lnot(\phi\,\mathtt u\,\psi)\ |\ n\in\omega\right\}
$$
infer $\theta \to \lnot(\phi\,\mathtt U\,\psi)$.

\end{itemize}

Modus ponens and necessitation have the same meaning as in any modal logic with multiple modal operators, see for instance \cite{Str92}. Rules R3 and R4 reflect the $\Leftarrow$ part
of the characterization of $\mathtt u$ and $\mathtt U$.

\subsection{Basic proof theory}

The inference relation $\vdash$ is defined as follows:

\begin{definition}

We say that an $L([1],[\omega],\mathtt u,\mathtt U)$-formula $\phi$ is a \emph{theorem} and write $\vdash \phi$ iff there exists at most countably infinite sequence of $L([1],[\omega],\mathtt u,\mathtt U)$-formulas
$\phi_0,\dots,\phi_\alpha$ (the ordering type of $\phi_0,\dots,\phi_\alpha$ is $\alpha+1$, where $\alpha$ is any countable ordinal) such that $\phi_\alpha=\phi$ and for all $\beta\leqslant \alpha$, $\phi_\beta$ is instance of some axiom, or $\phi_\beta$ can be obtained from some previous members of the sequence by application of some inference rule.
\hfill $\square$

\end{definition}

\begin{definition}

Suppose that $T$ is any $L([1],[\omega],\mathtt u,\mathtt U)$-theory and that $\phi$ is any $L([1],[\omega],\mathtt u,\mathtt U)$-formula. We say that $\phi$ is syntactical consequence of $T$ (or that $\phi$ is deducible or derivable from $T$) and write $T\vdash \phi$ iff there exists at most countably infinite sequence of $L([1],[\omega],\mathtt u,\mathtt U)$-formulas
$\phi_0,\dots,\phi_\alpha$ such that $\phi_\alpha=\phi$ and for all $\beta\leqslant \alpha$, $\phi_\beta$ is instance of some axiom, $\phi_\beta\in T$, or $\phi_\beta$ can be obtained from some previous members of the sequence by application of some inference rule, where the application of necessitation is restricted to theorems.
\hfill $\square$
\end{definition}

An immediate consequence of previous two definitions is the fact that structural rules cut ($T\vdash \Gamma$, $\Gamma\vdash\phi$ implies $T\vdash \phi$) and weakening ($T\vdash \phi$ implies $T,\psi\vdash \phi$) are true for the introduced consequence relation $\vdash$. The soundness theorem ($T\vdash \phi$ implies $T\models\phi$) can be straightforwardly proved by the induction on the length of the inference. Let us prove deduction theorem for $L([1],[\omega],\mathtt u, \mathtt U)$.

\begin{theorem}[Deduction theorem]
$T\vdash \phi\to\psi$ iff $T,\phi\vdash \psi$.

\end{theorem}

\begin{proof}
The $\Leftarrow$ part is a straightforward consequence of weakening and modus ponens. We will prove the converse implication by induction on the length of the inference.

If $\psi$ is an axiom instance or $\psi\in T$, then $T\vdash \psi$, so since $T\vdash \psi\to(\phi\to\psi)$ (A1), by modus ponens $T\vdash \psi\to \psi$. If $\psi=\phi$, then by A1, $T\vdash \phi\to\phi$.

Suppose that $\psi$ is a theorem. Then, $\vdash [a]\psi$, so by weakening we have that $T\vdash [a]\psi$, so $T\vdash \phi\to [a]\psi$. Thus, we have verified that the $\Rightarrow$ part is preserved
under the application of necessitation.

Suppose that $T\vdash \phi \to (\theta\to \lnot \psi_2)$ and $T\vdash \phi \to(\theta\to \bigvee_{k=0}^n[1]^k(\lnot \psi_1\lor \psi_2)\lor [1]^{n+1}\lnot\psi_2)$ for all $n\in \omega$. Since $p\to(q\to r)$ is
equivalent to $p\land q \to r$, we have that $T\vdash \phi\land \theta \to \lnot \psi_2$ and $T\vdash \phi\land \theta \to \bigvee_{k=0}^n[1]^k(\lnot \psi_1\lor \psi_2)\lor [1]^{n+1}\lnot\psi_2)$. By $\mathtt u$-rule,
$T\vdash \phi\land \theta \to \lnot (\psi_1\,\mathtt u\,\psi_2)$, i.e. $T\vdash \phi\to(\theta\to \lnot(\psi_1\,\mathtt u\,\psi_2))$. Hence, we have verified that the $\Rightarrow$ part is preserved under the application of the
$\mathtt u$-rule. The remaining step (application of the $\mathtt U$-rule) can be proved similarly.
\end{proof}

\subsection{Lindenbaum's theorem}

Roughly speaking, the standard maximization argument in the proof of Lindenbaum's theorem goes as follows: we start with a consistent theory and in consecutive steps we extend it with a single formula or with its negation (depending which choice preserves consistency), until we have exhausted all formulas. Due to the presence of infinitary inference rules, we additionally have to check in each step whether the current formula that is incompatible with the current theory can be derived by application of R3 or R4. If that is the case, we have to block at least one premise. Detailed construction is given below.

\begin{lemma}
\label{lemma Lindenbaum 1}
Suppose that $T$ is a consistent theory and that $T\vdash \lnot (\theta\to \lnot(\phi\,\mathtt u\,\psi))$. Then, $T,\lnot(\theta\to\lnot\psi)$ is consistent, or there exists a nonnegative integer $m$ such that $T, \lnot\left(\theta\to\bigvee_{k=0}^m[1]^k(\lnot\phi\lor\psi)\lor [1]^{m+1}\lnot\psi\right)$ is consistent.

\end{lemma}

\begin{proof}
If $T,\lnot(\theta\to\lnot\psi)$ and $T,\lnot\left(\theta\to\bigvee_{k=0}^m[1]^k(\lnot\phi\lor\psi)\lor [1]^{m+1}\lnot\psi\right)$ are inconsistent for all $m\in \omega$, then $T\vdash \theta\to\lnot \psi$ and
$T\vdash\theta\to \bigvee_{k=0}^m [1]^k(\lnot\phi\lor\psi)\lor [1]^{m+1}\lnot\psi$ for all $m\in \omega$, so by R3, $T\vdash \theta\to\lnot(\phi\,\mathtt u\,\psi)$, which contradicts
the fact that $T$ is consistent.
\end{proof}

\begin{lemma}
\label{lemma Lindenbaum 2}
Suppose that $T$ is a consistent theory and that $T\vdash \lnot(\theta\to\lnot(\phi\,\mathtt U\,\psi))$. Then, $T,\lnot(\theta\to\lnot(\phi\,\mathtt u\,\psi))$ is consistent, or there exists  a nonnegative integer $m$ such that
$T, \lnot\left(\theta\to\bigvee_{k=0}^m[\omega]^k\lnot\mathtt g(\phi\land\lnot\psi)\lor [\omega]^{m+1}\lnot(\phi\,\mathtt u\,\psi)\right)$ is consistent.
\end{lemma}

\begin{proof}
Similar as the proof of Lemma \ref{lemma Lindenbaum 1}, with the use of R4 instead of R3.
\end{proof}

\begin{theorem}[Lindenbaum's theorem]
\label{theorem Lindenbaum}
Every consistent theory can be maximized, i.e. extended to a complete theory.
\end{theorem}

\begin{proof}
Let $T$ be a consistent theory and let $\langle \phi_n\ |\ n\in\omega\rangle$ be one enumeration of the set of all $L([1],[\omega],\mathtt u,\mathtt U)$-formulas. We will inductively define the sequence
$\langle T_n\ |\ n\in\omega\rangle$ of theories as follows:

\begin{enumerate}

\item $T_0=T$;

\item If $T_n$ is compatible with $\phi_n$ ($T_n\cup \{\phi_n\}$ is consistent), then let $T_{n+1}=T_n\cup \{\phi_n$\};

\item If $T_n$ is incompatible with $\phi_n$ ($T_n\cup \{\phi_n\}$ is inconsistent) and $\phi_n\neq \theta\to\lnot(\psi_1\,\mathtt u\,\psi_2),\theta\to\lnot(\psi_1\,\mathtt U\,\psi_2)$, then let
$T_{n+1}=T_n\cup \{\lnot \phi_n\}$;

\item If $T_n$ is incompatible with $\phi_n$ and $\phi_n=\theta\to\lnot(\psi_1\,\mathtt u\,\psi_2)$, then let $T_{n+1}=T_n \cup\{\lnot\phi_n,\lnot\left(
\theta\to \bigvee_{k=0}^m[1]^k(\lnot\phi\lor\psi)\lor [1]^{m+1}\lnot\psi\right)\}$, where $m$ is the smallest nonnegative integer such that $T_{n+1}$ is consistent. If there is no such $m$, then by Lemma \ref{lemma Lindenbaum 1}, $T_n \cup \{\lnot \phi_n, \lnot(\theta\to \lnot\psi)\}$ is consistent, so let $T_{n+1}=T_n\cup\{\lnot \phi_n, \lnot(\theta\to \lnot\psi)\}$;

\item If $T_n$ is incompatible with $\phi_n$ and $\phi_n=\theta\to\lnot(\psi_1\,\mathtt U\,\psi_2)$, then let $T_{n+1}=T_n\cup\{\lnot\phi_n,\lnot\left(
\theta\to \bigvee_{k=0}^m[\omega]^k\lnot\mathtt g(\phi\land\lnot\psi)\lor [\omega]^{m+1}\lnot(\phi\,\mathtt u\,\psi)\right)\}$, where $m$ is the le\-a\-st nonnegative integer such that $T_{n+1}$ is consistent. If there is no such $m$, then by Lemma \ref{lemma Lindenbaum 2}, $T_n\cup\{\lnot \phi_n, \lnot(\theta\to \lnot(\phi\,\mathtt u\,\psi))\}$ is consistent, so let $T_{n+1}=T_n\cup\{\lnot \phi_n, \lnot(\theta\to \lnot(\phi\,\mathtt u\,\psi))\}$.

\end{enumerate}
Note that each $T_n$ is consistent. Let $T_\omega=\bigcup_{n\in\omega}T_n$. Clearly, $T_\omega \vdash \phi$ or $T_\omega\vdash\lnot\phi$ for any formula $\phi$. It remains to show consistency of $T_\omega$. In order to do so, it is sufficient to show that $T_\omega$ is deductively closed. Case (2) ensures that all axiom instances and all theorems are in $T_\omega$. Since necessitation can be applied only on theorems, $T_\omega$ is closed under its application. It remains to show that $T_\omega$ is closed under modus ponens, $\mathtt u$-rule and $\mathtt U$-rule.

MP: Let $\phi,\phi\to\psi\in T_\omega$. Then, there exist $k,m\in \omega$ so that $\phi\in T_k$ and $\phi\to\psi\in T_m$. Let $\psi=\phi_n$ and let $l=k+m+n+1$. By construction of $T_\omega$, $T_k,T_m,T_{n+1}\subseteq T_l$, so by modus ponens, $T_l\vdash \psi$. Since $T_l$ is consistent, $\psi=\phi_n$ and $T_n\subseteq T_l$, $T_n$ and $\psi$ are compatible, so $T_{n+1}=T_n,\psi$, i.e. $\psi\in T_\omega$.

$\mathtt u$-rule: Let $\{\theta\to\lnot\psi\}\cup\{\theta\to \bigvee_{i=0}^m[1]^i(\lnot\phi\lor\psi)\lor [1]^{m+1}\lnot\psi\ |\ m\in\omega\}\subseteq T_\omega$ and let $\theta\to\lnot(\phi\,\mathtt u\,\psi)=\phi_k$. Suppose that
$\theta\to\lnot(\phi\,\mathtt u\,\psi)\notin T_\omega$. Then, $T_{k+1}=T_k,\lnot \phi_k,\lnot(\theta\to \lnot\psi)$ or
$T_{k+1}=T_k,\lnot\phi_k,\lnot(\theta\to \bigvee_{i=0}^m[1]^i(\lnot\phi\lor\psi)\lor [1]^{m+1}\lnot\psi)$. Now, for sufficiently large index $n$ we have that $T_n$ contains both $\theta_1$ and $\lnot \theta_1$,
where $\theta_1=\theta\to\lnot \psi$ or  $\theta_1=\theta\to \bigvee_{i=0}^m[1]^i(\lnot\phi\lor\psi)\lor [1]^{m+1}\lnot\psi$, which contradicts consistency of $T_n$.

$\mathtt U$-rule: Similarly as the previous case.
\end{proof}

\subsection{Completeness theorem}

Due to the fact that we have established validity of Lindenbaum's theorem for $L([1],[\omega],\mathtt u,\mathtt U)$, it remains to show that each complete theory $T$ is satisfiable. Thus,
trough the rest of this section $T$ will be a complete theory.

\begin{definition}

Let $T$ be a complete theory. The canonical model $\xi_T$ is defined by
$$
\xi_T(k,l,p)=1\ \Leftrightarrow_{\rm def}\ T\vdash [\omega]^k[1]^lp.
$$
\hfill $\square$
\end{definition}

We will break down the proof of the fact that $\xi_T\models_{\langle 0,0\rangle}T$ into the following five lemmas:

\begin{lemma}
\label{lemma completeness 1}
Let $T\vdash \phi\,\mathtt u\,\psi$. Then, $T\vdash \psi$ or there exists a nonnegative integer $m$ such that
$T\vdash \bigwedge_{i=0}^m[1]^i(\phi\land\lnot\psi)\land [1]^{m+1}\psi$.
\end{lemma}

\begin{proof}
Contrary to the assumption of the lemma, let $T\not\vdash \psi$ and $T\not\vdash \bigwedge_{i=0}^m[1]^i(\phi\land\lnot\psi)\land [1]^{m+1}\psi$ for all $m\in \omega$. Since $T$ is complete,
$T\vdash \lnot\psi$ and $T\vdash \bigvee_{i=0}^m[1]^i(\lnot\phi\lor\psi)\lor [1]^{m+1}\lnot\psi$ for all $m\in \omega$, so by R3 (for $\theta$ we can chose any axiom instance),
$T\vdash \lnot(\phi\,\mathtt u\,\psi)$, which contradicts consistency of $T$.
\end{proof}

\begin{lemma}
\label{lemma completeness 2}
Let $T\vdash \phi\,\mathtt U\,\psi$. Then, $T\vdash \phi\,\mathtt u\,\psi$ or there exists a nonnegative integer $m$ such that
$T\vdash \bigwedge_{i=0}^m[\omega]^i\mathtt g(\phi\land\lnot\psi)\land [\omega]^{m+1}(\phi\,\mathtt u\,\psi)$.
\end{lemma}

\begin{proof}
Contrary to the assumption of the lemma, let $T\not\vdash \phi\,\mathtt u\,\psi$ and $T\not\vdash \bigwedge_{i=0}^m[\omega]^i\mathtt g(\phi\land\lnot\psi)\land [\omega]^{m+1}\psi$ for all $m\in \omega$. Since $T$ is complete,
$T\vdash \lnot(\phi\,\mathtt u\,\psi)$ and $T\vdash \bigvee_{i=0}^m[\omega]^i\lnot\mathtt g(\phi\land\lnot\psi)\lor [\omega]^{m+1}\lnot(\phi\,\mathtt u\,\psi)$ for all $m\in \omega$, so by R4 (for $\theta$ we can chose any axiom instance),
$T\vdash \lnot(\phi\,\mathtt U\,\psi)$, which contradicts consistency of $T$.
\end{proof}

\begin{lemma}
\label{lemma completeness 3}
Let $\phi$ be any $L([1],[\omega])$-formula. Then, for all $\langle k,l\rangle\in \omega\times \omega$,
$$
\xi_T\models_{\langle k,l\rangle}\phi\ {\rm iff}\ T\vdash [\omega]^k[1]^l\phi.
$$
\end{lemma}

\begin{proof}
We will use induction on the complexity of $L([1],[\omega])$-formulas. The case of propositional letters is covered by definition of $\xi_T$.

\begin{eqnarray*}
\xi_T\models_{\langle k,l\rangle}\lnot\phi &\Leftrightarrow& \xi_T\not\models_{\langle k,l\rangle} \phi\\
&\stackrel{IH}{\Leftrightarrow}& T\not\vdash [\omega]^k[1]^l\phi \\
&\Leftrightarrow& T\vdash \lnot[\omega]^k[1]^l\phi\\
&\stackrel{A3}{\Leftrightarrow}& T\vdash[\omega]^k[1]^l\lnot\phi;
\end{eqnarray*}

\begin{eqnarray*}
\xi_T\models_{\langle k,l\rangle}\phi\land\psi &\Leftrightarrow& \xi_T\models_{\langle k,l\rangle}\phi\ {\rm and}\
\xi_T\models_{\langle k,l\rangle}\psi\\
&\stackrel{IH}{\Leftrightarrow}& T\vdash [\omega]^k[1]^l\phi\ {\rm and}\ T\vdash [\omega]^k[1]^l\phi\\
&\Leftrightarrow& T\vdash([\omega]^k[1]^l\phi) \land ([\omega]^k[1]^l\psi)\\
&\stackrel{A4}{\Leftrightarrow}&T\vdash [\omega]^k[1]^l(\phi\land\psi);
\end{eqnarray*}

\begin{eqnarray*}
\xi_T\models_{\langle k,l\rangle}[1]\phi &\Leftrightarrow& \xi_T\models_{\langle k,l+1\rangle}\phi\\
&\stackrel{IH}{\Leftrightarrow}& T\vdash [\omega]^k[1]^{l+1}\phi\\
&\Leftrightarrow& T\vdash T\vdash [\omega]^k[1]^l[1]\phi;
\end{eqnarray*}

\begin{eqnarray*}
\xi_T\models_{\langle k,l\rangle}[\omega]\phi &\Leftrightarrow& \xi_T\models_{\langle k+1,0\rangle}\phi\\
&\stackrel{IH}{\Leftrightarrow}& T\vdash [\omega]^{k+1}\phi\\
&\Leftrightarrow& T\vdash [\omega]^k[\omega]\phi\\
&\stackrel{A2}{\Leftrightarrow}& T\vdash [\omega]^k[1]^l[\omega]\phi.
\end{eqnarray*}
\end{proof}

\begin{lemma}
\label{lemma completeness 4}
Let $\phi$ be any $L([1],[\omega],\mathtt u)$-formula. Then, for all $\langle k,l\rangle\in \omega\times \omega$,
$$
\xi_T\models_{\langle k,l\rangle}\phi\ {\rm iff}\ T\vdash [\omega]^k[1]^l\phi.
$$
\end{lemma}

\begin{proof}
As in the proof of the previous lemma, we will use the induction on the complexity of $L([1],[\omega],\mathtt u)$-formulas. By Lemma \ref{lemma completeness 3}, we only need to consider the case of $\mathtt u$-formulas.

Let $\xi_T\models_{\langle k,l\rangle} \phi\,\mathtt u\,\psi$. Then, $\xi_T\models_{\langle k,l\rangle}\psi$ or there exists a nonnegative integer $m$ such that
$\xi_T\models_{\langle k,l\rangle}\bigwedge_{i=0}^m[1]^i(\phi\land\lnot\psi)\land [1]^{m+1}\psi$. By induction hypothesis,
$T\vdash [\omega]^k[1]^l\psi$ or $T\vdash [\omega]^k[1]^l \bigwedge_{i=0}^m[1]^i(\phi\land\lnot\psi)\land [1]^{m+1}\psi$. By A4, A5 and A7, $T\vdash [\omega]^k[1]^l(\phi\,\mathtt u\,\psi)$.

Conversely, let $T\vdash [\omega]^k[1]^l(\phi\,\mathtt u\,\psi)$. By A4 and Lemma \ref{lemma completeness 1}, $T\vdash [\omega]^k[1]^l\psi$ or there exists a nonnegative integer $m$ such that
$T\vdash [\omega]^k[1]^l\bigwedge_{i=0}^m[1]^i(\phi\land\lnot\psi)\land [1]^{m+1}\psi$. By induction hypothesis, $\xi_T\models_{\langle k,l\rangle}\psi$ or
$\xi_T\models_{\langle k,l\rangle}\bigwedge_{i=0}^m[1]^i(\phi\land\lnot\psi)\land [1]^{m+1}\psi$, so by definition of $\models$, $\xi_T\models_{\langle k,l\rangle}\phi\,\mathtt u\,\psi$.
\end{proof}

\begin{lemma}
\label{lemma completeness 5}
Let $\phi$ be any $L([1],[\omega],\mathtt u,\mathtt U)$-formula. Then, for all $\langle k,l\rangle\in \omega\times \omega$,
$$
\xi_T\models_{\langle k,l\rangle}\phi\ {\rm iff}\ T\vdash [\omega]^k[1]^l\phi.
$$
\end{lemma}

\begin{proof}
A straightforward modification of the proof of the previous lemma based on lemmas \ref{lemma completeness 2} and \ref{lemma completeness 4}.
\end{proof}

\begin{corollary}
$\xi_T\models_{\langle 0,0\rangle}\phi$ iff $T\vdash \phi$.
\end{corollary}

\begin{proof}
An immediate consequence of Lemma \ref{lemma completeness 5}.
\end{proof}

Hence, we have proved the fact that each consistent theory is satisfiable, which concludes the proof of strong completeness theorem for $L([1],[\omega],\mathtt u,\mathtt U)$.

\subsection{Representation of zero-time transitions}

The main technical idea behind choosing $\omega^2$ as a model of the time flow is to separate states of the system and transitions from one state to the next. More precisely, time instants of the form $\langle k,0\rangle$ are reserved for temporal description of different states of a system, while transitions from $\langle k,0\rangle$ to $\langle k+1,0\rangle$ are temporally described throughout the $k$-th $\omega$-stick $\{\langle k,n\rangle\ |\ n\in\omega\}$. Recall that the state of a system in any time instant $\mathbf t$ is described by formulas satisfied in $\mathbf t$.

As it was described in \cite{GMM99}, any transition can occur within a closed time interval. If order to model this phenomenon, we need to provide the following two things:

\begin{itemize}

\item Any formula $\phi$ that is satisfied in all $\langle k,n\rangle$, $n\geqslant n_0$, should also be satisfied in $\langle k+1,0\rangle$;

\item Any change can occur only once.

\end{itemize}
This leads to additional transition axioms:

\begin{itemize}

\item[TR1] $\mathtt g\phi \to [\omega]\phi;$

\item[TR2] $(\phi\land [1]\lnot\phi)\to \mathtt g[1]\lnot\phi$.
\end{itemize}

Clearly, the strong completeness theorem holds for the extended system as well since both TR1 and TR2 can be seen as $L([1],[\omega],\mathtt u,\mathtt u)$-theories.

\section{Decidability}

In the process of showing decidability of the SAT problem for $L([1],[\omega]  ,\mathtt u,\mathtt U)$-formulas and determination
of the complexity, we will modify the argument presented by A. P. Sistla and E. M. Clarke in \cite{SC85}. By $Var(\phi)$ will be denoted the
set of all propositional variables appearing in $\phi$, while by ${\rm Sub}(\phi)$ will be denoted the set of all subformulas of $\phi$.
Note that
$$
|{\rm Sub}(\phi)|\leqslant {\rm lenght}(\phi),
$$
where $|{\rm Sub}(\phi)|$ is the cardinal number of ${\rm Sub}(\phi)$ and ${\rm lenght}(\phi)$ is the number of
symbols in $\phi$. Moreover, let
$$
{\rm Sub}(\phi,\mathbf t,\xi)=\{\psi\in {\rm Sub}(\phi)\ |\ \xi\models_\mathbf t\psi\}.
$$

\begin{lemma}
\label{deleting}
With notation as before, suppose that $\phi$, $\xi$, $\eta$, $\mathbf r$ and $\mathbf s$ satisfy the following conditions:

\begin{enumerate}

\item $\mathbf r<\mathbf s$;

\item ${\rm Sub}(\phi,\mathbf r,\xi)={\rm Sub}(\phi,\mathbf s,\xi)$;

\item $\eta(\mathbf t,p)=\xi(\mathbf t,p)$ for all $\mathbf t<\mathbf r$ and all $p\in Var$;

\item $\eta(r_1+i,r_2+j,p)=\xi(s_1+i,s_2+j,p)$ for all $i,j\in\omega$ and all $p\in Var$.

\end{enumerate}
Then, the following hold:

\begin{enumerate}

\item[(a)] ${\rm Sub}(\phi,\mathbf t,\xi)={\rm Sub}(\phi,\mathbf t,\eta)$ for all $\mathbf t<\mathbf r$;

\item[(b)] ${\rm Sub}(\phi,\langle s_1+i,s_2+j\rangle, \xi)={\rm Sub}(\phi, \langle r_1+i,r_2+j\rangle,\eta)$ for all $i,j\in\omega$.

\end{enumerate}

\end{lemma}

\begin{proof}
Note that (b) is an immediate consequence of Definition \ref{satisfiability relation} and the fourth condition in the statement of the lemma.
Therefore, it remains to prove (a). It turns out that a somewhat stronger claim is easier to prove. Namely, let $A$ be the set of all formulas with variables from $Var(\phi)$ that satisfy (a).
We claim that $Var(\phi)\subseteq A$ and that $A$ is closed under $\lnot$, $\land$, $[1]$, $[\omega]  $, $\mathtt u$ and $\mathtt U$.
Moreover,
since
$$
\psi\,\mathtt u\, \theta \Leftrightarrow \theta \lor \bigvee_{n\in\omega}\bigwedge_{k=0}^n[1]^k(\psi\land \lnot\theta)\land [1]^{n+1}\theta
$$
and
$$
\psi\,\mathtt U\,\theta \Leftrightarrow \psi\,\mathtt u\,\theta \lor \bigvee_{n\in\omega}\bigwedge_{k=0}^n[\omega]  ^k\mathtt g(\psi\land\lnot\theta)
\land [\omega]  ^{n+1}(\psi\,\mathtt u\,\theta),
$$
for the verification of the lemma it is sufficient to prove that $Var(\phi)\subseteq A$ and that $A$ is closed under $\lnot$, $\land$, $[1]$ and
$[\omega]  $.

By (3), all propositional letters
form $Var(\phi)$ satisfy (a). Moreover, if $\lnot \psi, \theta_1\land \theta_2\in {\rm Sub}(\phi)$ and $\psi$, $\theta_1$ and $\theta_2$
satisfy (a), then by Definition \ref{satisfiability relation} immediately follows that $\lnot\psi$ and $\theta_1\land\theta_2$ also satisfy (a).

Suppose that $[1]\psi\in {\rm Sub}(\phi)$ and that $\psi$ satisfies (a). We need to prove that $\xi\models_\mathbf t[1] \psi$ iff
$\eta\models_{\mathbf t}[1] \psi$ for all $\mathbf t<\mathbf r$. There are two relevant cases:

\begin{itemize}

\item $\langle t_1,t_2+1\rangle<\mathbf r$. Then,
\begin{eqnarray*}
\xi\models_\mathbf t[1] \psi &\Leftrightarrow& \xi\models_{\langle t_1,t_2+1\rangle}\psi\\
&\Leftrightarrow& \eta\models_{\langle t_1,t_2+1\rangle}\psi\ \ (\psi\in A\ {\rm and}\ \langle t_1,t_2+1\rangle<\mathbf r)\\
&\Leftrightarrow& \eta\models_\mathbf t[1] \psi.
\end{eqnarray*}

\item $\langle t_1,t_2+1\rangle=\mathbf r$ (note that this case cannot occur if $\mathbf r$ is a limit, i.e. if $r_2=0$).
Then,
\begin{eqnarray*}
\xi\models_\mathbf t[1] \psi &\Leftrightarrow& \xi\models_\mathbf r\psi\\
&\Leftrightarrow& \xi\models_\mathbf s\psi\ \ (\rm follows\ from (2))\\
&\Leftrightarrow& \eta\models_\mathbf r\psi\ \ (\rm follows\ from (4))\\
&\Leftrightarrow& \eta\models_\mathbf t[1]\psi.
\end{eqnarray*}

\end{itemize}

Suppose that $[\omega]  \,\psi$ is a subformula of $\phi$ and that $\psi$ satisfies (a). There are three relevant cases:

\begin{itemize}

\item $\langle t+1,0\rangle<\mathbf r$. Then,
\begin{eqnarray*}
\xi\models_{\mathbf t}[\omega]  \psi &\Leftrightarrow& \xi\models_{\langle t_1+1,0\rangle}\psi\\
&\Leftrightarrow& \eta\models_{\langle t_1+1,0\rangle}\psi\ \ (\psi\in A\ {\rm and}\ \langle t_1+1,0\rangle<\mathbf r)\\
&\Leftrightarrow& \eta\models_{\mathbf t}[\omega]  \psi;
\end{eqnarray*}

\item Let $\langle t_1+0,0\rangle=\mathbf r$. Then,
\begin{eqnarray*}
\xi\models_{\mathbf t}[\omega]  \psi &\Leftrightarrow& \xi\models_{\mathbf r}\psi\\
&\Leftrightarrow& \xi\models_{\mathbf s}\psi\ \ ({\rm follows\ from}\ (2))\\
&\Leftrightarrow& \eta\models_{\mathbf r }\psi\ \ ({\rm follows\  from}\ (4))\\
&\Leftrightarrow& \eta\models_{\langle t_1+1,0\rangle}\psi\ \ (\langle t_1+1,0\rangle=\mathbf r)\\
&\Leftrightarrow& \eta\models_{\mathbf t}[\omega]   \psi;
\end{eqnarray*}

\item Let $\langle t_1+1,0\rangle>\mathbf r$. Since $\mathbf t<\mathbf r$, it must be $t_1= r_1$, so
\begin{eqnarray*}
\xi \models_{\mathbf t}[\omega]  \psi &\Leftrightarrow& \xi\models_{\mathbf r}[\omega]  \psi\\
&\Leftrightarrow& \xi\models_{\mathbf s}[\omega]  \psi\ \ ({\rm follows\ from}\ (2))\\
&\Leftrightarrow& \xi\models_{\langle s_1+1,0\rangle}\psi\\
&\Leftrightarrow& \eta\models_{\langle r_1+1,0\rangle}\psi\ \ ({\rm follows\ from}\ (4))\\
&\Leftrightarrow& \eta\models_{\langle t_1+1,0\rangle}\psi\ \ (r_1=t_1)\\
&\Leftrightarrow& \eta\models_{\mathbf t}[\omega]  \psi.
\end{eqnarray*}

\end{itemize}

\end{proof}

\begin{lemma}[Outer loop]
\label{outer loop}

With notation as before, suppose that $\phi$, $\xi$, $\eta$, $k$ and $m$ ($k,m\in\omega$ and $m>0$) satisfy the following conditions:

\begin{enumerate}

\item  ${\rm Sub}(\phi,\langle k, 0\rangle,\xi)={\rm Sub}(\phi,\langle k+m, 0\rangle,\xi)$;

\item  $\xi(\mathbf t,p)=\eta(\mathbf t,p)$ for all $\mathbf t<\langle k+m,0\rangle$ and all $p\in Var$;

\item  $\eta(\mathbf s,p)=\eta(s_1+m,s_2,p)$ for all $\mathbf s\geqslant \langle k,0\rangle$ and all $p\in Var$;

\item If $\psi\,\mathtt U\,\theta\in{\rm Sub}(\phi)$ and $\xi\models_{\langle k,0\rangle}\psi\,\mathtt U\,\theta$,
then there exists $\langle k,0\rangle\leqslant \mathbf r
<\langle k+m,0\rangle$ such that $\xi\models_{\mathbf r}\theta$.

\end{enumerate}

Then, the following hold:

\begin{itemize}

\item[(a)] ${\rm Sub}(\phi,\mathbf t,\xi)={\rm Sub}(\phi,\mathbf t,\eta)$ for all $\mathbf t<\langle k+m,0\rangle$;

\item[(b)] ${\rm Sub}(\phi, \mathbf s, \eta)={\rm Sub}(\phi, \langle s_1+m,s_2\rangle, \eta)$ for all $\mathbf s\geqslant \langle k,0\rangle$.

\end{itemize}

\end{lemma}

\begin{proof}
\noindent  We will prove (a) and (b) simultaneously by induction on complexity of formulas. Firstly, (a) and (b) are obviously true for all propositional variables from $Var(\phi)$ and they are
preserved by negation and conjunction (this is an immediate consequence of Definition \ref{satisfiability relation}).

Suppose that $[1] \psi\in {\rm Sub}(\phi)$ and that $\psi$ satisfies both (a) and (b). Let $\mathbf t<\langle k+m,0\rangle$. Then,

\begin{eqnarray*}
\xi\models_{\mathbf t}[1] \psi&\Leftrightarrow& \xi\models_{\langle t_1,t_2+1\rangle}\psi\\
&\Leftrightarrow&\eta\models_{\langle t_1,t_2+1\rangle}\psi \ \ (\langle t_1,t_2+1\rangle<\langle k+m,0\rangle
\ {\rm and}\ \psi\ {\rm satisfies}\ (a))\\
&\Leftrightarrow& \eta\models_{\mathbf t}[1] \psi.
\end{eqnarray*}
Let $\langle k,0\rangle \leqslant \mathtt t<\langle k+m,0\rangle$. Then,
\begin{eqnarray*}
\eta\models_{\mathbf t}[1] \psi&\Leftrightarrow& \eta\models_{\langle t_1,t_2+1\rangle}\psi\\
&\Leftrightarrow&\eta\models_{\langle t_1+m,t_2+1\rangle}\psi \ \ (\langle t_1,t_2+1\rangle<\langle k+m,0\rangle
\ {\rm and}\ \psi\ {\rm satisfies}\ (b))\\
&\Leftrightarrow& \eta\models_{\langle t_1+m,t_2\rangle}[1] \psi.
\end{eqnarray*}

Suppose that $[\omega]   \psi$ is a subformula of $\phi$ and that $\psi$ satisfies both (a) and (b). Let $\mathbf t<\langle k+m,0\rangle$.
If $\langle t_1+1,0\rangle<\langle k+m,0\rangle$, then
\begin{eqnarray*}
\xi\models_{\mathbf t}[\omega]  \psi &\Leftrightarrow&\xi\models_{\langle t_1+1,0\rangle}\psi\\
&\Leftrightarrow& \eta\models_{\langle t_1+1,0\rangle}\psi\\
&\Leftrightarrow& \eta\models_{\mathbf t}[\omega]  \psi.
\end{eqnarray*}
If $\langle t_1+1,0\rangle=\langle k+m,0\rangle$, then
\begin{eqnarray*}
\xi\models_{\mathbf t}[\omega]   \psi &\Leftrightarrow& \xi\models_{\langle t_1+1,0\rangle}\psi\\
&\Leftrightarrow& \xi\models_{\langle k+m,0\rangle}\psi\\
&\Leftrightarrow& \xi\models_{\langle k,0\rangle}\psi\ \ ({\rm by\ (1)})\\
&\Leftrightarrow& \eta\models_{\langle k,0\rangle}\psi\ \ ({\rm by\ (2)})\\
&\Leftrightarrow& \eta\models_{\langle k+m,0\rangle}\psi\ \ ({\rm by\ (b)})\\
&\Leftrightarrow& \eta\models_{\mathbf t}[\omega]  \psi.
\end{eqnarray*}
Let $\langle k,0\rangle\leqslant \mathbf t<\langle k+m,0\rangle$.
If $\langle t_1+1,0\rangle<\langle k+m,0\rangle$, then
\begin{eqnarray*}
\eta\models_{\mathbf t}[\omega]  \psi &\Leftrightarrow&\eta\models_{\langle t_1+1,0\rangle}\psi\\
&\Leftrightarrow& \eta\models_{\langle t_1+1+m,0\rangle}\psi\\
&\Leftrightarrow& \eta\models_{\langle t_1+m+1,0\rangle}\psi\\
&\Leftrightarrow& \eta\models_{\langle t_1+m,t_2\rangle}[\omega]  \psi.
\end{eqnarray*}
If $\langle t_1+1,0\rangle=\langle k+m,0\rangle$, then
\begin{eqnarray*}
\eta\models_{\mathbf t}[\omega]   \psi &\Leftrightarrow& \eta\models_{\langle t_1+1,0\rangle}\psi\\
&\Leftrightarrow& \eta\models_{\langle k+m,0\rangle}\psi\\
&\Leftrightarrow& \eta\models_{\langle k+2m,0\rangle}\psi\ \ ({\rm by\ (3)})\\
&\Leftrightarrow& \eta\models_{\langle t_1+m+1,0\rangle}\psi\\
&\Leftrightarrow& \eta\models_{\langle t_1+m,t_2\rangle}[\omega]  \psi.
\end{eqnarray*}

Suppose that $\psi\,\mathtt u\,\theta$ is a subformula of $\phi$ and that $\psi$ and $\theta$ satisfy both (a) and (b).
Let $\mathbf t<\langle k+m,0\rangle$. Then, $\langle t_1,t_2+i\rangle<\langle k+m,0\rangle$ for all $i<\omega$,
so $\xi\models_{\langle t_1,t_2+i\rangle}\theta$ iff $\eta\models_{\langle t_1,t_2+i\rangle}\theta$ for all $i<\omega$ and
$\xi\models_{\langle t_1,t_2+i\rangle}\psi\land\lnot\theta$ iff $\eta\models_{\langle t_1,t_2+i\rangle}\psi\land\lnot\theta$ for all $i<\omega$.
Consequently, $\xi\models_{\mathbf t}\psi\,\mathtt u\,\theta$ iff
$\eta\models_{\mathbf t}\psi\,\mathtt u\,\theta$. Similarly, if $\langle k,0\rangle\leqslant \mathbf t<\langle k+m,0\rangle$,
then $\eta\models_{\mathbf t}\psi\,\mathtt u\,\theta$ iff $\eta\models_{\langle t_1+m,t_2\rangle}\psi\,\mathtt u\,\theta$

Suppose that $\psi\,\mathtt U\,\theta$ is a subformula of $\phi$ and that $\psi$ and $\theta$ satisfy both (a) and (b). Then, for all
$\mathbf t<\langle k+m,0\rangle$, $\xi\models_{\mathbf t}\theta$ iff $\eta\models_{\mathbf t}\theta$
and $\xi\models_{\mathbf t}\psi\land\lnot\theta$ iff $\xi\models_{\mathbf t}\psi\land\lnot\theta$.

Let $\xi\models_{\mathbf t}\psi\,\mathtt U\,\theta$ for some $\mathbf t<\langle k+m,0\rangle$. By (4), there exists $\mathbf t\leqslant \mathbf r
<\langle k+m,0\rangle$ such that $\xi\models_{\mathbf r}\theta$. Since $[\mathbf t, \langle k+m,0\rangle)_{\rm lex}$ is a well ordering, we can
chose minimal $\mathbf r\in [\mathbf t, \langle k+m,0\rangle)_{\rm lex}$ such that $\xi\models_{\mathbf r}\theta$. Now we have that
$\xi\models_{\mathbf s}\psi\land \lnot \theta$ for all $\mathbf t\leqslant \mathbf s<\mathbf r$, so by induction hypothesis and
Definition \ref{satisfiability relation}, $\eta\models_{\mathbf t}\psi\,\mathtt U\,\theta$.

Let $\eta\models_{\mathbf t}\psi\,\mathtt U\,\theta$ for some $\mathbf t<\langle k+m,0\rangle$. Then, there exists $\mathbf r\geqslant \mathbf t$
such that $\eta\models_{\mathbf r}\theta$ and $\eta\models_{\mathbf s}\psi\land \lnot\theta$ for all $\mathbf t\leqslant \mathbf s <\mathbf r$.
Since $[\mathbf t,\infty)_{\rm lex}$ is a well ordering, we can chose minimal $\mathbf r$. Moreover, (b) implies that minimal $\mathbf r$ is strictly lesser than
$\langle k+m,0\rangle$, so we have that $\xi\models_{\mathbf t}\psi\,\mathtt U\,\theta$.

It remains to prove (b) for  $\psi\,\mathtt U\,\theta$, provided that $\psi$ and $\theta$ satisfy
both (a) and (b). By (b) and Definition \ref{satisfiability relation},
$\eta\models_{\mathbf s}\theta$ iff $\eta\models_{\langle s_1+m,s_2\rangle}\theta$
and $\eta\models_{\mathbf s}\psi\land\lnot \theta$ iff $\eta\models_{\langle s_1+m,s_2\rangle}\psi\land\lnot\theta$ for all
$\mathbf s\in [\langle k,0\rangle,\langle k+m,0\rangle)_{\rm lex}$. As we have already proved, $\eta\models_{\mathbf s}\psi\,\mathtt U\,\theta$ iff
$\xi\models_{\mathbf s}\psi\,\mathtt U\,\theta$ for all $\mathbf s\in [\langle k,0\rangle,\langle k+m,0\rangle)_{\rm lex}$. As a consequence, $\eta$ satisfies (4).

Let $\langle k,0\rangle\leqslant\mathbf t<\langle k+m,0\rangle$ and let $\eta\models_{\mathbf t}\psi\,\mathtt U\,\theta$. By (4), there exists
$\mathbf t\leqslant\mathbf r <\langle k+m,0\rangle$ such that $\eta\models_{\mathbf r}\theta$ and $\eta\models_{\mathbf s}\psi\land \lnot\theta$ for all
$\mathbf s\in[\mathbf t,\mathbf r)_{\rm lex}$. By induction hypothesis,
$\eta\models_{\langle r_1+m,r_2\rangle}\theta$ and $\eta\models_{\langle s_1+m,s_2\rangle}\psi\land \lnot\theta$ for all
$\mathbf s\in[\mathbf t,\mathbf r)_{\rm lex}$, so $\eta\models_{\langle t_1+m,t_2\rangle}\psi\,\mathtt U\,\theta$. The converse implication can be proved similarly.
\end{proof}

\begin{lemma}[Inner loop]
\label{inner loop}

With notation as before, suppose that $\phi$, $\xi$, $\eta$, $k$, $l$ and $m$ ($k,l,m\in\omega$ and $m>0$) satisfy the following conditions:

\begin{enumerate}

\item  ${\rm Sub}(\phi,\langle k, l\rangle,\xi)={\rm Sub}(\phi,\langle k, l+m\rangle,\xi)$;

\item  $\xi(\mathbf t,p)=\eta(\mathbf t,p)$ for all $\mathbf t<\langle k,l+m\rangle$, $\mathbf t\geqslant \langle k+1,0\rangle$ and all $p\in Var$;

\item  $\eta(k,i,p)=\eta(k,i+m,p)$ for all $i\geqslant l$ and all $p\in Var$;

\item If $\psi\,\mathtt u\,\theta\in{\rm Sub}(\phi)$ and $\xi\models_{\langle k,l\rangle}\psi\,\mathtt u\,\theta$,
then there exists $\langle k,l\rangle\leqslant \mathbf r
<\langle k,l+m\rangle$ such that $\xi\models_{\mathbf r}\theta$.

\end{enumerate}

Then, the following hold:

\begin{itemize}

\item[(a)] ${\rm Sub}(\phi,\mathbf t,\xi)={\rm Sub}(\phi,\mathbf t,\eta)$ for all $\mathbf t<\langle k,l+m\rangle$ and all $\mathbf t\geqslant \langle k+1,0\rangle$;

\item[(b)] ${\rm Sub}(\phi, \mathbf s, \eta)={\rm Sub}(\phi, \langle s_1,s_2+m\rangle, \eta)$ for all
$\mathbf s\in [\langle k,l\rangle,\langle k+1,0\rangle)_{\rm lex}$.

\end{itemize}
\end{lemma}

\begin{proof}

We will prove (a) and (b) simultaneously by induction on complexity of formulas. As before, the statement is obviously true for
propositional variables and its validity is preserved under negation and conjunction. Furthermore, (2) implies that
${\rm Sub}(\phi,\mathbf t,\xi)={\rm Sub}(\phi,\mathbf t,\eta)$ for all $\mathbf t\geqslant \langle k+1,0\rangle$.
Similarly as in the proof of Lemma \ref{outer loop} we can show that (4) holds for $\eta$ as well.

Suppose that $[1] \psi$ is a subformula of $\phi$ and that $\psi$ satisfies both (a) and (b). Let $\mathbf t<\langle k,l+m\rangle$.
If $t_2+1<l+m$, then by induction hypothesis $\xi\models_{\mathbf t}[1] \psi$ iff $\eta\models_{\mathbf t}[1] \psi$. Let
$t_2+1=l+m$. Then
\begin{eqnarray*}
\xi\models_{\mathbf t}[1] \psi &\Leftrightarrow& \xi\models_{\langle k,l+m\rangle}\psi\\
&\Leftrightarrow& \xi\models_{\langle k,l\rangle}\psi\\
&\Leftrightarrow& \eta\models_{\langle k,l+m\rangle}\psi\\
&\Leftrightarrow& \eta\models_{\mathbf t}[1] \psi.
\end{eqnarray*}
Let $\langle k,l\rangle\leqslant \mathbf t<\langle k+1,0\rangle$. Then,
\begin{eqnarray*}
\eta\models_{\mathbf t}[1] \psi &\Leftrightarrow& \eta\models_{\langle t_1,t_2+1\rangle}\psi\\
&\Leftrightarrow& \eta\models_{\langle t_1,t_2+1+m\rangle}\psi\\
&\Leftrightarrow& \eta\models_{\langle t_1,t_2+m+1\rangle}\psi\\
&\Leftrightarrow& \eta\models_{\langle t_1,t_2+m\rangle}[1]\psi.
\end{eqnarray*}

Suppose that $[\omega]  \psi$ is a subformula of $\phi$ and that $\psi$ satisfies (a) and (b). Let $\mathbf t<\langle k,l+m\rangle$.
If $t_1+1<k$, then (a) can be straightforwardly verified for $[\omega]  \psi$ by induction hypothesis. If $t_1+1=k$, then
\begin{eqnarray*}
\xi\models_{\mathbf t}[\omega]  \psi &\Leftrightarrow& \xi\models_{\langle k,0\rangle}\psi\\
&\Leftrightarrow& \eta\models_{\langle k,0\rangle}\psi\\
&\Leftrightarrow& \eta\models_{\mathbf t}[\omega]  \psi.
\end{eqnarray*}
Let $t_1=k$. Then,
\begin{eqnarray*}
\xi\models_{\mathbf t}[\omega]  \psi &\Leftrightarrow& \xi\models_{\langle k+1,0\rangle}\psi\\
&\Leftrightarrow& \eta\models_{\langle k+1,0\rangle}\psi\\
&\Leftrightarrow& \eta\models_{\mathbf t}[\omega]  \psi.
\end{eqnarray*}

Let $\langle k,l\rangle\leqslant \mathbf t<\langle k+1,0\rangle$. Then, by Definition \ref{satisfiability relation}, $\eta\models_{\mathbf t}[\omega]  \psi$ iff
$\eta\models_{\langle k, l+i\rangle}[\omega]  \psi$ for all $i<\omega$. Hence $[\omega]  \psi$ satisfies (b).

Suppose that $\psi\,\mathtt u\,\theta$ is a subformula of $\phi$ and that both $\psi$ and $\theta$ satisfy (a) and (b).
Let $\mathbf t<\langle k,l+m\rangle$ and $\xi\models_{\mathbf t}\psi\,\mathtt u\,\theta$. The only nontrivial case is when $t_1=k$ and $t_2\geqslant l$.
Then, there exists $n<\omega$ such that $t_2+n<l+m$, $\xi_{\langle k,t_2+n\rangle}\models \theta$ and $\xi_{\langle k,l+i\rangle}\psi\land\lnot\theta$
for all $i\in\{0,\dots,n-1\}$. By induction hypothesis, we can replace $\xi$ with $\eta$, so $\eta\models_{\mathbf t}\psi\,\mathtt u\,\theta$. The converse
implication can be proved similarly.

Let $\langle k,l\rangle \leqslant \mathbf t<\langle k+1,0\rangle$. By induction hypothesis,
for all $i<\omega$ we have that $\eta\models_{\langle k,i\rangle}\theta$ iff $\eta\models_{\langle k,i+m\rangle}\theta$
and $\eta\models_{\langle k,i\rangle}\psi\land \lnot\theta$ iff $\eta\models_{\langle k,i+m\rangle}\psi\land\lnot\theta$.
Hence, by Definition \ref{satisfiability relation}, $\psi\,\mathtt u\,\theta$ satisfies (b).

Suppose that $\psi\,\mathtt U\,\theta$ is a subformula of $\phi$ and that both $\psi$ and $\theta$ satisfy (a) and (b). Let $\mathbf t<\langle, k,l+m\rangle$ and
$\xi\models_{\mathbf t}\psi\,\mathtt U\,\theta$. Then, there exists $\mathbf r\geqslant \mathbf t$ such that $\xi\models_{\mathbf r}\theta$ and
$\xi\models_{\mathbf s}\psi\land\lnot\theta$ for all $\mathbf t\leqslant \mathbf s<\mathbf r$. If $\mathbf r\geqslant \langle k, l+m\rangle$, then
by (3) we can replace $\xi$ with $\eta$, hence $\eta\models_{\mathbf t}\psi\,\mathtt U\,\theta$. If $\langle k,l\rangle\leqslant \mathbf r<\langle k,l+m\rangle$,
then
\begin{eqnarray*}
\xi\models_{\mathbf t}\psi\,\mathtt U\,\theta &\Leftrightarrow& \xi\models_{\langle k,l\rangle}\psi\,\mathtt u\,\theta\\
&\Leftrightarrow& \eta\models_{\langle k,l\rangle}\psi\,\mathtt u\,\theta\\
&\Leftrightarrow& \eta\models_{\mathbf t}\psi\,\mathtt U\,\theta.
\end{eqnarray*}
Finally, if $\mathbf r<\langle k,l\rangle$, then $\eta\models_{\mathbf t}\psi\,\mathtt U\,\theta$ immediately follows from induction hypothesis. The converse implication, as well as
the fact that $\psi\,\mathtt U\,\theta$ satisfies (b), can be proved similarly.
\end{proof}

\begin{theorem}[Periodicity]
\label{periodicity}
Let $\xi\models_{\mathbf t}\phi$. Then, there exist a model $\eta$ and positive integers $k$, $m$, $l_i$ and $n_i$, $i=1,\dots,k+m$ with the following properties:

\begin{enumerate}

\item $\eta\models_{\langle 0,0\rangle}\phi$;

\item $\eta(i,j,p)=\eta(i+m,j,p)$ for all $i> k$ and all $j\in\omega$;

\item $\eta(i,j,p)=\eta(i,j+n_i,p)$ for all $i\leqslant k+m$ and all $j>l_i$;

\item $\max(m,n_1,\cdots,n_m)\leqslant{\rm length}(\phi)\cdot 2^{{\rm length}(\phi)}$;

\item $\max(k,l_1,\dots,l_{k+m})\leqslant 2^{{\rm length}(\phi)}$.

\end{enumerate}
\end{theorem}

\begin{proof}
Starting with a model $\xi$, we will build desired model $\eta$ gradually. Let
$$
\eta_0(i,j,p)=\xi(i+t_1,j+t_2,p),\ \ i,j\in\omega, \ p\in Var.
$$
Clearly, $\eta_0\models_{\langle i,j\rangle}\psi$ iff $\xi\models_{\langle i+t_1,j+t+2\rangle}\psi$ for all $\psi\in For$,
hence $\eta_0$ satisfies (1).

Let $\mathbf r\sim \mathbf s$ iff ${\rm Sub}(\phi, \mathbf r,\eta_0)={\rm Sub}(\phi, \mathbf s,\eta_0)$. Clearly, $\sim$ is an equivalence relation
on $\omega \times \omega$ with finitely many equivalence classes. The number of classes is bounded by $2^{{\rm length}(\phi)}$, since there is at most
${\rm length}(\phi)$ subformulas of $\phi$. Consequently, a simple application of the pigeonhole principle shows that there exist positive integers $i$ and $j$,
$i<j$ so that $\langle i,0\rangle_{/\sim}$ is infinite, $\langle i,0\rangle\sim \langle j,0\rangle$
and
$\eta_0$, $\langle i,0\rangle$ and $\langle j,0\rangle$ satisfy condition (4) in Lemma \ref{outer loop}. Now we apply Lemma \ref{deleting} to ``delete"
all time instants $\langle r,0\langle\sim \langle i,0\rangle$, $r\neq i,j$ that do not violate the condition (4) in Lemma \ref{outer loop}. Since there are at most
${\rm length}(\phi)$ subformulas of $\phi$, the starting index $k$ of the outer loop (loop on first indices) is bounded by the number of classes of $\sim$, i.e.
$k\leqslant 2^{{\rm length}(\phi)}$ and the outer period $m$ is bounded by the product of the number of classes of $\sim$ and the number
of subformulas of $\phi$, i.e. $m\leqslant {\rm length}(\phi)\cdot 2^{{\rm length}(\phi)}$.

Let $\eta_1$ be a model such that $\eta_0$, $\eta_1$, $\langle k+1,0\rangle$, $\langle k+m+1,0\rangle$ satisfy conditions of Lemma \ref{outer loop}. It is easy to see that $\eta_1$ satisfies conditions (1) and (2) of the theorem. Similarly, using lemmas \ref{deleting} and \ref{inner loop}, we can transform
$\eta_1$ and obtain positive integers $l_1,\dots,l_{k+m}$, $n_1,\dots,n_m$ and a model $\eta$ that satisfy (1)--(4).
\end{proof}

\begin{theorem}
\label{complexity}

The satisfiability of $L([1],[\omega]  ,\mathtt u,\mathtt U)$-formulas is PSPACE complete.
\end{theorem}

\begin{proof}
On the one hand, in \cite{SC85} is shown that the satisfiability of $L([1],\mathtt u)$-formulas is PSPACE complete, so the satisfiability of $L([1],[\omega],\mathtt u,\mathtt U)$ formulas is at least PSPACE hard. For the other direction, we will construct a nondeterministic Turing machine that determines satisfiability and uses polynomial space with respect to the length of a given formula.
Before the start of the actual description of the TM, we shall clarify the notation and terminology:

\begin{itemize}

\item $k,m$: nonnegative integers that have the same meaning as the corresponding numbers in Theorem \ref{periodicity} ($k$ is the size of the initial segment, while $m$ is the size of the outer loop);

\item $k_{\rm local}$, $m_{\rm local}$: local versions of $k$ and $m$  (they are locally restricted to the current inner loop $\{\langle i,n\rangle\ |\ n\in\omega\}$);

\item $S_{\rm start}$: subformulas guessed to be true at $\langle 0,0\rangle$;

\item $S_{\rm present}$: subformulas guessed to be true at the present moment;

\item $S_{[1]}$: subformulas guessed to be true at the next time instant $\langle i,j+1\rangle$;

\item $S_{[\omega]}$: subformulas guessed to be true at $\omega$-jump $\langle i+1,0\rangle$;

\item $S_{\rm in}$: subformulas guessed to be true at the beginning of the inner period $\langle i, k_{\rm loc}\rangle$;

\item $S_{\mathtt u}$: the set of all $\mathtt u$-formulas from $S_{\rm in}$;

\item $S_{\rm out}$: subformulas guesed at the beginning of the beginning outer period $\langle k,0\rangle$;

\item $S_{\mathtt U}$: the set of all $\mathtt U$-formulas from $S_{\rm out}$;

\item Any of sets $S_*$ is said to satisfy {\it Boolean consistency} iff for all $\psi,\theta\in Sub(\phi)$ the following is true:

\begin{itemize}

\item $\psi\in S_*$ or $\lnot \psi\in S_*$;

\item $\psi\in S_*\Leftrightarrow \lnot\psi\notin S_*$;

\item $\psi\land\theta\in S_*\Leftrightarrow \psi,\theta \in S_*$;

\end{itemize}

\item We say that  $S_{\rm present}$, $S_{[1]}$ and $S_{[\omega]}$ are {\it properly linked} iff for all $\psi,\theta\in Sub(\phi)$ the following conditions are satisfied:

\begin{itemize}

\item $[1]\psi\in S_{\rm present}\Leftrightarrow \psi\in S_{[1]}$;

\item $[\omega]\psi \in S_{\rm present}\Leftrightarrow \psi\in S_{[\omega]}$;

\item $\psi\,\mathtt u\,\theta\in S_{\rm present}$ iff $\theta\in S_{\rm present}$ or $\psi\land \lnot\theta\in S_{\rm present}$ and $\psi\,\mathtt u\,\theta\in S_{[1]}$;

\item $\psi\,\mathtt U\,\theta\in S_{\rm present}$ iff $\theta\in S_{\rm present}$ or $\mathtt g(\psi\land \lnot \theta)\in S_{\rm present}$ and $\psi\,\mathtt U\,\theta\in S_{[\omega]}$.

\end{itemize}

\end{itemize}
Our TM works as follows:

\begin{itemize}

\item[] input $\phi$;

\item[] guess $k$, $m$, $S_{\rm start}$ and  $S_{\rm out}$;

\item[] check Boolean consistency of $S_{\rm start}$ and  $S_{\rm out}$; if fails return {false};

\item[] if $\phi\notin S_{\rm start}$ return false;

\item[] construct $S_{\mathtt U}$;

\item[] $S_{\rm present}:=S_{\rm start}$;

\item[] for $i=0$ to $k+m-1$ do;

\begin{itemize}

\item[] guess $k_{\rm loc}$, $m_{\rm loc}$, $S_{\rm in}$;

\item[] if $i<k+m-1$ guess $S_{[\omega]}$; else
$S_{[\omega]}:=S_{\rm out}$;

\item[] check Boolean consistency of $S_{\rm in}$ and
$S_{[\omega]}$; if fails return false;

\item[] construct $S_{\mathtt u}$;

\item[] for $j=0$ to $k_{\rm loc}+m_{\rm loc}-1$ do;

\begin{itemize}

\item[] if $j<k_{\rm loc}+m_{\rm loc}-1$ guess $S_{[1]}$; else
$S_{[1]}:=S_{\rm in}$;

\item[] check Boolean consistency of $S_{[1]}$; if fails return false;

\item[] check whether $S_{\rm present}$, $S_{[1]}$ and $S_{[\omega]}$ are properly linked; if fails return false;

\item[] if $j\geqslant k_{\rm loc}$ then for all $\psi\,\mathtt u\,\theta\in S_\mathtt u$ check whether $\theta\in S_{\rm present}$; if pass, delete $\psi\,\mathtt u\,\theta$ from $S_\mathtt u$;

\item[] if $k\leqslant i<k+m-1$ then for all $\psi\,\mathtt U\,\theta\in S_\mathtt U$ check whether $\theta\in S_{\rm present}$; if pass, delete $\psi\,\mathtt U\,\theta$ from $S_\mathtt U$;

\item[] $S_{\rm present}:=S_{[1]}$;

\item[] next $j$;

\end{itemize}

\item[] if $S_\mathtt u \neq \emptyset$ return false;

\item[] $S_{\rm present}:=S_{[\omega]}$;

\item[] next $i$;

\end{itemize}

\item[] if $S_{\mathtt U}\neq\emptyset$ return false;

\item[] end.

\end{itemize}

It is easy to see that just described TM uses polynomial space with respect to ${\rm length}(\phi)$, so satisfiability problem for $L([1],[\omega],\mathtt u,\mathtt U)$-formulas is at most PSPACE hard. Since PSPACE is both upper and lower complexity bound, we have our claim.
\end{proof}

\section{Conclusion}

The starting point for this paper was prior research on modelling
zero-time transitions using infinitesimals \cite{GMM99}. We
expanded  linear time logic to the formalism adequate for
modelling those transitions. We proposed an axiomatization for the
logic, proved strong completeness and determined the complexity
for the satisfiability procedure. Since the logic is not compact,
the axiomatization contains infinitary rules of inference. As a
topic for further research, we propose developing generalization
of the logic $L([1],[\omega],\mathtt u,\mathtt U)$ to any
$\omega^n$, $n> 2$, with the addition of new unary temporal
operators $[\omega^2],\dots,[\omega^{n-1}]$ and binary temporal
operators $\mathtt U_1,\dots, \mathtt U_n$. Here $\mathtt
U_1=\mathtt u$ and $\mathtt U_2=\mathtt U$. Besides theoretical,
such generalization can be used for temporal description of
complex systems with many subsystem layers.

\subsection{Acknowledgements}
The authors are partially supported by Serbian ministry of education and science through grants III044006, III041103, ON174062 and TR36001.

\end{document}